\documentclass[11pt,a4paper]{article}

\usepackage[T1]{fontenc}
\usepackage{lmodern}
\usepackage[margin=1in]{geometry}
\usepackage{amsmath,amssymb,amsthm,mathtools,microtype}
\usepackage[colorlinks=true,linkcolor=blue,urlcolor=blue,citecolor=blue]{hyperref}
\usepackage[nameinlink,noabbrev]{cleveref}
\usepackage{enumitem}

\allowdisplaybreaks
\setlist[itemize]{topsep=3pt,itemsep=2pt}
\numberwithin{equation}{section}

\newtheorem{theorem}{Theorem}[section]

\newtheorem{lemma}[theorem]{Lemma}

\theoremstyle{remark}
\newtheorem{remark}[theorem]{Remark}

\newcommand{\HH}{\mathcal H}
\newcommand{\KK}{\mathcal K}
\newcommand{\cA}{\mathcal A}
\newcommand{\DD}{\mathbb D}
\newcommand{\TT}{\mathbb T}
\newcommand{\CC}{\mathbb C}
\newcommand{\norm}[1]{\left\lVert #1\right\rVert}
\newcommand{\ip}[2]{\left\langle #1,#2\right\rangle}
\newcommand{\Ree}{\operatorname{Re}}

\title{On the abstract approach to spectral constants:\\
  a proof of the Clou\^atre--Ostermann--Ransford conjecture}
\author{Catalin Badea, Ryan O'Loughlin, and Jani Virtanen}
\date{1 September 2026}

\begin{document}

\maketitle

\begin{abstract}
Clou\^atre, Ostermann, and Ransford formulated an abstract version of
Crouzeix's conjecture involving a bounded unital homomorphism from a
uniform algebra into matrices and a
unital antilinear map.  They conjectured that contractivity of the
associated symmetrised map forces the homomorphism to have norm at most
two.  We prove this conjecture, in fact without requiring the antilinear
map to be contractive and for homomorphisms into the bounded operators on a Hilbert space. The proof
combines positivity of real parts, an operator-valued Herglotz theorem, and the perturbation lemma of Lorist and Schwenninger used in the recent proof of Crouzeix's conjecture.
\end{abstract}

\medskip
\noindent\textit{2020 Mathematics Subject Classification.}
Primary 47A25; Secondary 46J10, 47A12, 47L30.

\noindent\textit{Keywords.}
Crouzeix's conjecture, spectral set, uniform algebra, Herglotz
representation, dilation.

\section{Introduction}

Let \(T\in B(\HH)\), where \(\HH\) is a complex Hilbert space.  Its
\emph{numerical range} is
\[
 W(T)=\{\ip{Tx}{x}:x\in\HH,\ \norm{x}=1\}.
\]
%We take inner products to be linear in the first variable.  
The closure
\(\overline{W(T)}\) is a compact convex set containing \(\sigma(T)\), the spectrum of $T$.
Crouzeix's conjecture asserts that it is a \(2\)-spectral set for \(T\):
\begin{equation}
 \norm{r(T)}
 \leq 2\sup_{z\in W(T)}|r(z)|
 \label{eq:crouzeix}
\end{equation}
for every rational function \(r\) whose poles lie outside
\(\overline{W(T)}\).  Since \(\overline{W(T)}\) is convex, the rational
formulation is equivalent to the polynomial one by standard polynomial
approximation (Runge's theorem).

Delyon and Delyon proved that \(\overline{W(T)}\) is always a
\(C\)-spectral set, with \(C\) depending on the geometry of the numerical
range \cite{DelyonDelyon}.  Crouzeix proved the conjectured estimate for
\(2\times2\) matrices (see also \cite{BadeaCrouzeixDelyon}) and formulated the general conjecture
in \cite{CrouzeixIEOT}.  He subsequently
obtained the first universal bound, with constant \(11.08\)
\cite{CrouzeixJFA}.  Crouzeix and Palencia later reduced the universal
constant to \(1+\sqrt2\) \cite{CrouzeixPalencia}.  The constant two is
optimal, as is already shown by the nilpotent matrix
\(
 \left(\begin{smallmatrix}0&1\\0&0\end{smallmatrix}\right)
\).

Crouzeix's conjecture has attracted considerable attention. For background material and an account of earlier developments, see \cite{BadeaBeckermann}\footnote{A second edition of this survey is in preparation.} and \cite{BickelEtAl}.
At the time of writing, three independent proofs of the scalar
conjecture have recently appeared, due to Jin \cite{Jin}, Lorist and
Schwenninger \cite{LoristSchwenninger}, and Luo \cite{Luo}.  The
completely bounded, or matrix-valued, analogue remains open.

The purpose of this note is to settle a conjecture of Clou\^atre,
Ostermann, and Ransford \cite[Conjecture~1.2]{COR}, which arose from their
abstract approach to Crouzeix's conjecture.  The original conjecture
concerns a continuous unital homomorphism
\(\theta:\cA\to M_n(\CC)\) on a uniform algebra and a unital antilinear
contraction \(\alpha:\cA\to\cA\).  It predicts that
\[
 \left\|\frac12\bigl(\theta(\,\cdot\,)
       +\theta(\alpha(\,\cdot\,))^*\bigr)\right\|\leq1
 \quad\Longrightarrow\quad
 \norm{\theta}\leq2.
\]
Clou\^atre, Ostermann, and Ransford proved the bound \(1+\sqrt2\) without
the unitality assumption on \(\alpha\), and showed that this constant is
sharp in that setting.  Their extremal example satisfies
\(\alpha(1)=-1\).  They also proved that their conjecture implies
Crouzeix's conjecture.

Jin proved the abstract assertion when the finite-dimensional
commutative algebra \(\theta(\cA)\) is semisimple \cite{Jin}.  Lorist and
Schwenninger \cite{LoristSchwenninger} obtained a related abstract result when the symmetrised map
is completely positive; see \cite{PaulsenBook} for this notion.  Hartz and
McCarthy recently answered a related completely bounded question in the
special case of scalar-valued correction terms \cite{HartzMcCarthy}.

Our proof relies on the perturbation lemma introduced by Lorist and Schwenninger, which was the main new ingredient in their proof of Crouzeix's conjecture. The only other ingredients are two standard results: a unital contractive linear map defined on a uniform algebra preserves positivity of real parts, and every operator-valued function with positive real part admits a Herglotz--Naimark dilation.

\section{The abstract sharp bound}

Recall that a \emph{uniform algebra} on a compact Hausdorff space \(X\)
is a unital norm-closed subalgebra \(\cA\) of \(C(X)\), equipped with the
uniform norm.  All homomorphisms in this paper are complex-linear and
unital.  An antilinear map \(\alpha\) is called unital when
\(\alpha(1)=1\).

The following theorem is slightly stronger than the original
Clou\^atre--Ostermann--Ransford conjecture: the Hilbert space need not be
finite-dimensional, and \(\alpha\) need only be bounded, rather than
contractive.

\begin{theorem}[Abstract sharp bound]
\label{thm:abstract-sharp}
Let \(\cA\subset C(X)\) be a uniform algebra, let
\[
 \theta:\cA\longrightarrow B(\HH)
\]
be a bounded unital homomorphism, and let
\[
 \alpha:\cA\longrightarrow\cA
\]
be a bounded unital antilinear map.  Define
\begin{equation}
 \Phi(h)
 =\frac12\bigl(\theta(h)+\theta(\alpha(h))^*\bigr),
 \qquad h\in\cA.
 \label{eq:symmetrised-map}
\end{equation}
If \(\Phi\) is contractive, then
\begin{equation}
 \norm{\theta}\leq2.
 \label{eq:abstract-two}
\end{equation}
\end{theorem}

The map \(\Phi\) in \eqref{eq:symmetrised-map} is complex-linear: both
\(\alpha\) and the adjoint operation are antilinear.  Taking
\(\HH=\CC^n\) and assuming, as in \cite{COR}, that \(\alpha\) is a
contraction gives the original conjecture immediately.

\section{Proof of the main result}

We begin with two standard positivity facts.

\begin{theorem}[Unital contractions preserve positive real parts]
\label{lem:unital-real-positive}
Let \(\cA\subset C(X)\) be a uniform algebra and let
\(\Phi:\cA\to B(\HH)\) be complex-linear, unital, and contractive.  Then
\[
 \Ree h\geq0\ \text{on }X
 \quad\Longrightarrow\quad
 \Ree\Phi(h)\geq0.
\]
\end{theorem}

\begin{proof}
Fix a unit vector \(x\in\HH\) and define
\[
 \varphi_x(h)=\ip{\Phi(h)x}{x},\qquad h\in\cA.
\]
Since \(\varphi_x(1)=1\) and \(\norm{\varphi_x}\leq1\), we have
\(\norm{\varphi_x}=1\).  By the Hahn--Banach theorem,
\(\varphi_x\) extends to a norm-one functional
\(\widetilde\varphi_x\) on \(C(X)\), still satisfying
\(\widetilde\varphi_x(1)=1\).  Such a functional is a state on the
commutative \(C^*\)-algebra \(C(X)\), and is therefore positive.  Hence,
if \(\Ree h\geq0\), then
\[
 \Ree\ip{\Phi(h)x}{x}
 =\widetilde\varphi_x(\Ree h)
 \geq0.
\]
This holds for every unit vector \(x\), so \(\Ree\Phi(h)\geq0\).
\end{proof}

We shall use the following known dilation form of the operator-valued Herglotz
representation. We give a proof for completeness.

\begin{theorem}[Herglotz--Naimark dilation]
\label{thm:herglotz}
Let \(H:\DD\to B(\HH)\) be analytic and suppose that
\[
 H(0)=I,
 \qquad
 \Ree H(w)\geq0\quad(w\in\DD).
\]
Then there are a Hilbert space \(\KK\), an isometry
\(V:\HH\to\KK\), and a unitary \(U\in B(\KK)\) such that
\begin{equation}
 H(w)
 =V^*(I+wU^*)(I-wU^*)^{-1}V
 =I+2\sum_{n=1}^{\infty}w^nV^*U^{*n}V.
 \label{eq:herglotz-dilation}
\end{equation}
The series converges in operator norm, locally uniformly on \(\DD\).
\end{theorem}

\begin{proof}
The operator-valued Herglotz representation gives a positive
operator-valued measure \(\mu\) on \(\TT\), with \(\mu(\TT)=I\), such
that
\[
 H(w)=\int_{\TT}\frac{\zeta+w}{\zeta-w}\,d\mu(\zeta).
\]
There is no additional imaginary constant because \(H(0)=I\) is
self-adjoint.  By Naimark's dilation theorem, there are a spectral
measure \(E\) on \(\TT\), acting on a Hilbert space \(\KK\), and an
isometry \(V:\HH\to\KK\) such that
\[
 \mu(\Delta)=V^*E(\Delta)V
\]
for every Borel set \(\Delta\subset\TT\).  Put
\(U=\int_{\TT}\zeta\,dE(\zeta)\).  Since
\[
 \frac{\zeta+w}{\zeta-w}
 =\frac{1+w\overline\zeta}{1-w\overline\zeta}
 =1+2\sum_{n=1}^{\infty}w^n\overline\zeta^{,n},
\]
we obtain \eqref{eq:herglotz-dilation}.  Local operator-norm convergence
follows from \(\norm{U}=1\).
\end{proof}

We now isolate the transfer principle that contains the main argument.

\begin{theorem}[Positive-layer transfer]
\label{thm:positive-layer-transfer}
Let \(\cA\subset C(X)\) be a uniform algebra, let
\(\theta:\cA\to B(\HH)\) be a bounded unital homomorphism, and let
\(P:\cA\to B(\HH)\) be a bounded complex-linear map satisfying
\begin{equation}
 P(1)=2I
 \label{eq:layer-mass}
\end{equation}
and
\begin{equation}
 \Ree h\geq0\ \text{on }X
 \quad\Longrightarrow\quad
 \Ree P(h)\geq0.
 \label{eq:layer-positive}
\end{equation}
Suppose that, for every \(f\in\cA\) and every \(n\geq1\),
\begin{equation}
 \bigl[P(f^n)-\theta(f^n),\,\theta(f)^*\bigr]=0.
 \label{eq:positive-layer-commutation}
\end{equation}
Then \(\norm{\theta}\leq2\).
\end{theorem}

The proof of the above result uses the following Lorist–Schwenninger perturbation lemma [11, Lemma 1].

\begin{lemma}[The Lorist–Schwenninger perturbation lemma]
\label{lemma:LS}
Let $\HH$ and $\KK$ be Hilbert spaces. Let
$T\in B(\HH)$, let $Q\in B(\KK)$ with $\|Q\| \le 1$, and let
$V\in B(\HH,\KK)$ be an isometry.  Suppose that the operators
\begin{equation}
 E_n=2V^*Q^{*n}V-T^{*n},
 \qquad n\geq1,
 \label{eq:errors}
\end{equation}
satisfy
\begin{equation}
 M:=\sup_{n\geq1}\norm{E_n}<\infty
 \label{eq:uniform-errors}
\end{equation}
and
\begin{equation}
 E_nT=TE_n,
 \qquad n\geq1.
 \label{eq:commutation}
\end{equation}
Then $\norm{T} \le 2$.
\end{lemma}

\begin{proof}[Proof of Theorem~\ref{thm:positive-layer-transfer}]
Fix \(f\in\cA\) with \(\norm{f}_X\leq1\).  For \(w\in\DD\), define
the Cayley function
\[
 c_w=(1+wf)(1-wf)^{-1}
 =1+2\sum_{n=1}^{\infty}w^nf^n
\]
and put
\begin{equation}
 H_f(w)=\frac12P(c_w).
 \label{eq:H-from-layer}
\end{equation}
The inverse and the series belong to \(\cA\), by the Neumann series.
Moreover, \(H_f(0)=I\), and pointwise on \(X\) we have
\[
 \Ree c_w
 =\frac{1-|w|^2|f|^2}{|1-wf|^2}
 \geq0.
\]
It follows from \eqref{eq:layer-positive} that
\(\Ree H_f(w)\geq0\).  By \Cref{thm:herglotz}, there are a unitary
\(U\in B(\KK)\) and an isometry \(V:\HH\to\KK\) such that
\[
 H_f(w)=I+2\sum_{n=1}^{\infty}w^nV^*U^{*n}V.
\]
On the other hand, boundedness of \(P\) and
\eqref{eq:H-from-layer} give
\[
 H_f(w)=I+\sum_{n=1}^{\infty}w^nP(f^n).
\]
Comparison of Taylor coefficients yields
\begin{equation}
 2V^*U^{*n}V=P(f^n),
 \qquad n\geq1.
 \label{eq:layer-moments}
\end{equation}

Set \(R=\theta(f)^*\).  Since \(\theta\) is a homomorphism,
\[
 R^{*n}=\theta(f)^n=\theta(f^n).
\]
For \(n\geq1\), define
\[
 E_n
 =2V^*U^{*n}V-R^{*n}
 =P(f^n)-\theta(f^n).
\]
By \eqref{eq:positive-layer-commutation}, \(E_n\) commutes with \(R\).
The errors $E_n$ are uniformly bounded, since
\[
 \norm{E_n}
 \leq \norm{P}\,\norm{f^n}_X
      +\norm{\theta}\,\norm{f^n}_X
 \leq \norm{P}+\norm{\theta}.
\]
Thus all the hypotheses of the Lorist--Schwenninger perturbation Lemma \ref{lemma:LS}
%\cite[Lemma~1]{LoristSchwenninger} 
are satisfied, with \(T=R\) and
\(Q=U\).  Consequently,
\[
 \norm{\theta(f)}=\norm{R}\leq2.
\]
Taking the supremum over the closed unit ball of \(\cA\) proves the
claim.
\end{proof}

\begin{proof}[Proof of \Cref{thm:abstract-sharp}]
The map \(\Phi\) is complex-linear, unital, and contractive.  By
\Cref{lem:unital-real-positive}, it preserves positive real parts.
Consequently, \(P=2\Phi\) satisfies
\eqref{eq:layer-mass} and \eqref{eq:layer-positive}.

For \(f\in\cA\) and \(n\geq1\),
\begin{equation}
 P(f^n)-\theta(f^n)
 =\theta(\alpha(f^n))^*.
 \label{eq:abstract-defect}
\end{equation}
Because \(\cA\) is commutative, both \(\theta(\cA)\) and its adjoint
algebra
\[
 \theta(\cA)^*=\{\theta(h)^*:h\in\cA\}
\]
are commutative.  Hence
\[
 \bigl[\theta(\alpha(f^n))^*,\,\theta(f)^*\bigr]=0.
\]
The commutation condition \eqref{eq:positive-layer-commutation} follows,
and \Cref{thm:positive-layer-transfer} gives \(\norm{\theta}\leq2\).
\end{proof}

\begin{remark}
The boundedness of \(\alpha\), rather than its contractivity, is enough:
it ensures that the defect map in \eqref{eq:abstract-defect} is bounded.
\end{remark}

\bigskip

\noindent {\bf Acknowledgement}. The authors acknowledge the support of CDP C2EMPI and institutional partners for the R-CDP-24-004-C2EMPI project, the EU-COST actions and EPSRC (grant EP/Y008375/1).

\bigskip
\small

\noindent
\textsc{Catalin Badea}\\
Department of Mathematics and Statistics, University of Reading,
Reading, United Kingdom\\
Laboratoire Paul Painlev\'e, Universit\'e de Lille,
Lille, France\\[2mm]
\texttt{c.badea@reading.ac.uk}, \texttt{cbadea@univ-lille.fr}

\medskip

\noindent
\textsc{Ryan O'Loughlin}\\
Department of Mathematics and Statistics, University of Reading,
Reading, United Kingdom\\[2mm]
\texttt{r.d.oloughlin@reading.ac.uk}

\medskip

\noindent
\textsc{Jani Virtanen}\\
Department of Physics and Mathematics, University of Eastern Finland,
80100 Joensuu, Finland\\
University of Reading, Reading, United Kingdom\\
University of Helsinki, Helsinki, Finland\\[2mm]
\texttt{jani.virtanen@uef.fi},
\texttt{j.a.virtanen@reading.ac.uk},
\texttt{jani.virtanen@helsinki.fi}


\begin{thebibliography}{99}
\small
\raggedright
\setlength{\itemsep}{2pt}
\setlength{\parskip}{0pt}

\bibitem{BadeaBeckermann}
C.~Badea and B.~Beckermann,
Spectral sets,
in \emph{Handbook of Linear Algebra}, 2nd ed.,
CRC Press, Boca Raton, FL, 2014, Chapter~37, pp.~37-1--37-26.

\bibitem{BadeaCrouzeixDelyon}
C.~Badea, M.~Crouzeix, and B.~Delyon,
Convex domains and \(K\)-spectral sets,
\emph{Math. Z.} \textbf{252} (2006), no.~2, 345--365.

\bibitem{BickelEtAl}
K.~Bickel, P.~Gorkin, A.~Greenbaum, T.~Ransford,
F.~L.~Schwenninger, and E.~Wegert,
Crouzeix's conjecture and related problems,
\emph{Comput. Methods Funct. Theory} \textbf{20} (2020),
no.~3--4, 701--728.

\bibitem{COR}
R.~Clou\^atre, M.~Ostermann, and T.~Ransford,
An abstract approach to the Crouzeix conjecture,
\emph{J. Operator Theory} \textbf{90} (2023), no.~1, 209--221.

\bibitem{CrouzeixIEOT}
M.~Crouzeix,
Bounds for analytical functions of matrices,
\emph{Integral Equations Operator Theory} \textbf{48} (2004),
no.~4, 461--477.

\bibitem{CrouzeixJFA}
M.~Crouzeix,
Numerical range and functional calculus in Hilbert space,
\emph{J. Funct. Anal.} \textbf{244} (2007), no.~2, 668--690.

\bibitem{CrouzeixPalencia}
M.~Crouzeix and C.~Palencia,
The numerical range is a \((1+\sqrt2)\)-spectral set,
\emph{SIAM J. Matrix Anal. Appl.} \textbf{38} (2017),
no.~2, 649--655.

\bibitem{DelyonDelyon}
B.~Delyon and F.~Delyon,
Generalization of von Neumann's spectral sets and integral representation
of operators,
\emph{Bull. Soc. Math. France} \textbf{127} (1999), no.~1, 25--41.

\bibitem{HartzMcCarthy}
M.~Hartz and J.~E.~McCarthy,
From Clou\^atre--Ostermann--Ransford to Okubo--Ando,
arXiv:2606.02922v1 [math.FA], 2026.

\bibitem{Jin}
S.~Jin,
The numerical range is a \(2\)-spectral set,
\emph{Preprints} 2026, 202607.1919, version~4,
\href{https://doi.org/10.20944/preprints202607.1919.v4}
{doi:10.20944/preprints202607.1919.v4}.

\bibitem{LoristSchwenninger}
E.~Lorist and F.~L.~Schwenninger,
A solution to Crouzeix's conjecture,
arXiv:2608.03841v2 [math.CA], 2026,
\href{https://doi.org/10.48550/arXiv.2608.03841}
{doi:10.48550/arXiv.2608.03841}.

\bibitem{Luo}
Q.~Luo,
Two boundary-certificate proofs of the Crouzeix theorem: fiberwise
Fourier recurrence and defect--Gram telescoping,
\emph{Preprints} 2026, 202608.1661, version~2,
\href{https://doi.org/10.20944/preprints202608.1661.v2}
{doi:10.20944/preprints202608.1661.v2}.

\bibitem{PaulsenBook}
V.~Paulsen,
\emph{Completely Bounded Maps and Operator Algebras},
Cambridge Studies in Advanced Mathematics, vol.~78,
Cambridge University Press, Cambridge, 2002.

\end{thebibliography}
\end{document}